\documentclass[11pt]{amsart}

\usepackage{epigamath-runin2}

\usepackage[english]{babel}

\numberwithin{equation}{subsection}

\usepackage[shortlabels]{enumitem}
\usepackage{latexsym,amsmath, amscd, amssymb, amsthm, euscript, mathrsfs}
\usepackage[all]{xy}
\usepackage{url}

\newtheorem{thm}[subsection]{Theorem}
\newtheorem{cor}[subsection]{Corollary}
\newtheorem{lem}[subsection]{Lemma}
\newtheorem{prop}[subsection]{Proposition}

\theoremstyle{definition}
\newtheorem{defn}[subsection]{Definition}

\theoremstyle{remark}
\newtheorem{rem}[subsection]{Remark}
\newtheorem{example}[subsection]{Example}

\setlist[enumerate,1]{label={\rm(\roman*)}, ref={\rm\roman*}}

\newcommand{\mc}{\mathcal }
\newcommand{\scr}[1]{\mathbf{\EuScript{#1}}}
\newcommand{\logs}{\scr {L}og}

\DeclareMathOperator{\Sp}{Spec}
\DeclareMathOperator{\colim}{colim}
\DeclareMathOperator{\Der}{Der}
\DeclareMathOperator{\Hom}{Hom}
\DeclareMathOperator{\sgn}{sgn}
\DeclareMathOperator{\Qcoh}{Qcoh}
\DeclareMathOperator{\Ext}{Ext}
\DeclareMathOperator{\Et}{Et}
\DeclareMathOperator{\Cov}{Cov}
\DeclareMathOperator{\Eq}{Eq}

\newcommand{\id}{\mathrm{id}}
\newcommand{\dR}{\mathrm{dR}}
\newcommand{\gp}{\mathrm{gp}}
\newcommand{\qcoh}{\mathrm{qcoh}}
\newcommand{\aff}{\mathrm{aff}}
\newcommand{\op}{\mathrm{op}}
\newcommand{\Mod}{\mathrm{Mod}}
\newcommand{\dga}{\mathrm{dga}}
\newcommand{\pt}{\mathrm{pt}}
\newcommand{\gr}{\mathrm{gr}}

\newcommand{\mls}{\mathscr}

\newcommand{\Liset}{{\mathrm{Lis}\textrm{-}\mathrm{\acute{E}t}}}
\newcommand{\liset}{{\mathrm{lis}\textrm{-}\mathrm{\acute{e}t}}}
\newcommand{\et}{\mathrm{\acute{e}t}}
\newcommand{\mcH }{\mc H \mc H}

\newcommand{\lra}{\longrightarrow}

\makeatletter
\newcommand{\xrightrightarrows}[2][]{\mathrel{
 \raise.40ex\hbox{$
       \ext@arrow 3095\rightarrowfill@{\phantom{#1}}{#2}$}
 \setbox0=\hbox{$\ext@arrow 0359\rightarrowfill@{#1}{\phantom{#2}}$}
 \kern-\wd0 \lower.40ex\box0}}
\newcommand{\longrightrightarrows}{\xrightrightarrows{\hphantom{0pt}}} 

\newcommand{\longhookrightarrow}{\lhook\joinrel\longrightarrow}

\newcommand{\supth}[1]{\ensuremath{#1^{\mathrm{th}}}}

\EpigaVolumeYear{10}{2026} \EpigaArticleNr{7} \ReceivedOn{January 22, 2025}
\InFinalFormOn{September 24, 2025}
\AcceptedOn{November 14, 2025}

\title{Hochschild homology for log schemes}
\titlemark{Hochschild homology for log schemes}

\author{Martin Olsson}
\address{Department of Mathematics, 970 Evans Hall \#3840, University of California, Berkeley, CA 94720-3840, USA}
\email{martinolsson@berkeley.edu}

\authormark{M.~Olsson}

\AbstractInEnglish{We extend the notions of Hochschild and cyclic homology to morphisms from algebraic spaces to algebraic stacks.  Using this, we obtain generalizations of these homology theories to log schemes in the sense of Fontaine and Illusie.}



\acknowledgement{This work was partially supported by NSF grants DMS-1601940, DMS-1902251, DMS-2151946, and the Simons Collaboration on Perfection in Algebra, Geometry, and Topology.}

\begin{document}



\maketitle

\begin{prelims}

\DisplayAbstractInEnglish

\bigskip

\DisplayKeyWords

\medskip

\DisplayMSCclass

\end{prelims}


\newpage

\setcounter{tocdepth}{1}

\tableofcontents


\section{Introduction}

The purpose of this article is to define Hochschild and cyclic homology for log schemes in the sense of Fontaine and Illusie.   As in other work of the author on logarithmic geometry, the basic idea here is to view logarithmic Hochschild homology as a special case of the classical theory without log structures using the stacks $\logs $ introduced in \cite{O}.  

\subsection{}\label{R:1.1a}
The approach discussed here was originally presented in a lecture in 2018; see  \cite{Olssonlecture}. Subsequently, several other authors have defined Hochschild homology for log schemes, including \cite{binda2023hochschildkostantrosenberg, hablicsek2023logarithmic}.  The output of the theory developed here agrees with the one in \cite{hablicsek2023logarithmic}, as noted in that article.  The use of the stacks $\logs $, however, places the theory in a much more general setting of defining relative Hochschild homology for a morphism $X\rightarrow \mc S$, for $X$ an algebraic space and $\mc S$ an algebraic stack.  This is the setting for much of the work in this article.  The difference between our approach and that in \cite{binda2023hochschildkostantrosenberg} is akin to the comparison between the two approaches to the cotangent complex in \cite{MR2195148}.
    More precisely, we expect  that one should be able to define a map from the Hochschild homology complex in \cite{binda2023hochschildkostantrosenberg} to the log Hochschild homology complex defined in this paper, using the observation \cite[Remark~3.16]{binda2023hochschildkostantrosenberg} that the Hochschild homology defined in that paper can be defined using a suitable derived version of the log diagonal combined with the description in Theorem~\ref{T:4.2} below.  Combined with the HKR isomorphism and the comparison between the two log cotangent complexes in \cite[Corollary~8.34]{MR2195148}, we expect this map to be an isomorphism for integral morphisms of log schemes.  This in turn suggests  that the Hochschild homology defined in \cite{binda2023hochschildkostantrosenberg} could be defined by Kan extension from the theory defined in this article.  However, we have not worked out the many technical details involved with this potential approach.

\subsection{} 
Let $k$ be a ring and $M_k$ a fine log structure on $\Sp (k)$.  For a log flat morphism of fine log schemes
$$
f\colon (X, M_X)\lra (\Sp (k), M_k),
$$
 with underlying morphism of schemes quasi-compact and quasi-separated, we define certain $k$-modules
 $$
 HH_n((X, M_X)/(k, M_k)),
 $$
 which we refer to as the \emph{Hochschild homology groups of\, $(X, M_X)/(k, M_k)$}, as well as $k$-modules 
 $$
HC_n((X, M_X)/(k, M_k)),
$$
which we refer to as the \emph{cyclic homology groups of\, $(X, M_X)/(k, M_k)$}.

In the case when the log structures are trivial, our definitions agree with the definitions for schemes given  in \cite{W, WG}.

\begin{rem} Much of the discussion carries over without the flatness assumption.  However, in the non-flat case, it is better to consider derived versions of the theory.  We do not consider that here and instead restrict to the flat case.
\end{rem}

\subsection{} The Hochschild homology is obtained by taking cohomology of a complex $\mcH _{(X, M_X)/(k, M_k)}$ whose associated cohomology sheaves (with the usual reindexing) we will denote by $$\mcH _*((X, M_X)/(k, M_k)).$$  To define cyclic homology we need to use more structure on $\mcH  _*((X, M_X)/(k, M_k))$, as in the classical case of schemes.  We will show the following properties of Hochschild and cyclic homology:

\begin{enumerate}
\item\label{item1} (Multiplicative structure) The Hochschild homology sheaves $\mcH _*((X, M_X)/(k, M_k))$ form a commutative graded algebra with $\mcH _0((X, M_X)/(k, M_k)\simeq \mls O_X$ (as sheaves of $k$-algebras).

\item\label{item2} (Connes operator) There is a generalization of the Connes operator 
$$
B\colon \mcH _m((X, M_X)/(k, M_k))\lra \mcH _{m+1}((X, M_X)/(k, M_k))
$$
giving $\mcH _*((X, M_X)/(k, M_k))$ the structure of a commutative differential graded algebra.  Moreover, the induced map
$$
\mls O_X\lra \mcH _0((X, M_X)/(k, M_k))\overset{B}\lra \mcH _1((X, M_X)/(k, M_k))
$$
has a natural extension to a log derivation (in the sense of \cite[Definition~1.1.2]{Og}) defining an isomorphism
$$
\epsilon _1\colon \Omega ^1_{(X, M_X)/(k, M_k)}\lra \mcH _1((X, M_X)/(k, M_k)).
$$

\item\label{item3} (HKR isomorphism)  By the universal property of the log de Rham complex, the map $\epsilon _1$ extends to a morphism of commutative differential graded algebras
$$
\left(\Omega ^\bullet _{(X, M_X)/(k, M_k)}, d\right)\lra (\mcH _*((X, M_X)/(k, M_k)), B).
$$
If $(X, M_X)\rightarrow (\Sp (k), M_k)$ is log smooth, then this map is an isomorphism.

\item\label{item4} (SBI sequence) Cyclic and Hochschild homology are related by a long exact sequence 
$$
\cdots \lra HH_m((X, M_X)/(k, M_k))\lra  HC_m((X, M_X)/(k, M_k))\lra  HC_{m-2}((X, M_X)/(k, M_k))\lra  \cdots 
$$
generalizing the SBI sequence for schemes; see \cite[Theorem~2.2.1]{L}.

\item\label{item5} If $X = \Sp (A)$ is affine and $(X, M_X)$ is log smooth over $(k, M_k)$, then there is a spectral sequence abutting to $HC_*((X, M_X)/(k, M_k))$ with $E^1_{pq}$-term given by $\Omega ^{q-p}_{(X, M_X)/(k, M_k)}$ for $q\geq p\geq 0$ and $0$ otherwise.
If $k$ is a $\mathbb{Q}$-algebra, then this spectral sequence degenerates at the $E^2$-page and there is an isomorphism
$$
HC_m((X, M_X)/(k, M_k))\simeq \Omega ^m_{(X, M_X)/(k, M_k)}\oplus \bigoplus _{i\geq 1}H^{m-2i}_{\dR}((X, M_X)/(k, M_k)).
$$
\end{enumerate}

{\samepage
\begin{rem} 
  \leavevmode
  \begin{enumerate}
    \item The decomposition in~\eqref{item5} above can be described more canonically using Adams operations, which we discuss in Section~\ref{S:section9}.

    \item In~\eqref{item4} above we write the SBI sequence in the form most commonly found in the literature.  However, from the viewpoint of spaces with $S^1$-action, the sequence is more canonically written with $HC_{m-2}((X, M_X)/(k, M_k))$ replaced by $HC_{m-2}((X, M_X)/(k, M_k))\otimes _{\mathbb{Z}}H_1(S^1, \mathbb{Z})$, the choice of an orientation for $S^1$ giving an isomorphism (see for example \cite[Appendix~D.6]{L}).  Written this way the sequence in characteristic $0$ also becomes compatible with Adams operations as in \cite[Th\'eor\`eme~4.8]{L2}.
      \end{enumerate}
\end{rem}
}

\subsection{} The work in this article is based on our earlier results \cite{O} providing a dictionary between certain notions in logarithmic geometry with corresponding notions in the theory of stacks.  We summarize some of the basic aspects of this dictionary in Section~\ref{S:logreview}.  Using this, the definitions and results about Hochschild and cyclic homology for log schemes are obtained from more general definitions and results about Hochschild and cyclic homology for morphisms $X\rightarrow \mls S$ from  a scheme (or algebraic space) to an algebraic stack.  We do not write out these more general statements here in the introduction, but as noted above it is in this generality that we work for most of the article.

There is one point in this regard, however, that we wish to highlight.  In general, for a morphism $X\rightarrow \mls S$ from an algebraic space to an algebraic stack, the universal property of $\Omega ^1_{X/\mls S}$ is not captured by derivations alone like in the case of schemes.  In the logarithmic setting this manifests itself in the fact that a log derivation contains more information than an ordinary derivation.  In the case of schemes, the de Rham complex $\Omega ^\bullet _{X/\mls S}$ is the universal commutative differential graded algebra receiving an algebra map from $\mls O_X$.  In general, however, the universal property of $\Omega ^\bullet _{X/\mls S}$ is not expressed as simply.  The solution to this problem used in this article involves the introduction of a site $\mls C_{X/\mls S}$ which we discuss in Section~\ref{S:section5.5}.

\begin{rem} Similar techniques to the ones employed here should enable a study of logarithmic topological Hochschild homology.  A theory of logarithmic topological Hochschild homology, as well as logarithmic Hochschild homology, has been developed by Rognes, Sagave, and Schlichtkrull \cite{MR2544395, MR3412362, MR3760301} using a different approach.  As in Section~\ref{R:1.1a} we expect that the two approaches should be related by comparing suitable diagonal maps in the derived and ordinary setting.
\end{rem}

\subsection{Conventions}

All schemes and stacks considered in this article will be assumed to be locally finite-dimensional and locally Noetherian.

\subsection*{Acknowledgments}
Our approach to the foundations is based on a presentation by Peter Scholze of the usual theory for ordinary schemes which he gave at the Simons symposium on P-adic Hodge theory in May, 2017.  I thank Bhargav Bhatt, Lars Hesselholt, and M\'arton Hablicsek for helpful conversations and correspondence, and the referee for their useful suggestions and corrections.

\section{Log schemes and stacks}\label{S:logreview}

\subsection{}  Recall from \cite{O} that if $(S, M_S)$ is a fine log scheme, then there is an associated algebraic stack, denoted by $\logs _{(S, M_S)}$, over the category of $S$-schemes.  The fiber of $\logs _{(S, M_S)}$ over an $S$-scheme $f\colon X\rightarrow S$ is the groupoid of pairs $(M_X, f^b)$, where $M_X$ is a fine log structure on $X$ and $f^b\colon f^*M_S\rightarrow M_X$ is a morphism of log structures on $X$.   By \cite[Theorem~1.1]{O} the stack $\logs _{(S, M_S)}$ is an algebraic stack with finitely presented diagonal.  Giving a morphism of log schemes $(X, M_X)\rightarrow (S, M_S)$ is therefore equivalent to giving an ordinary morphism
$$
X\lra \logs _{(S, M_S)}.
$$

The stack $\logs _{(S, M_S)}$ is not of finite type in general, but we will primarily be interested in the case when $X$ is quasi-compact.  In this case a morphism $X\rightarrow \logs _{(S, M_S)}$ factors through a quasi-compact open substack $\mls U\subset \logs _{(S, M_S)}$.

\subsection{} Properties of morphisms of log schemes can be described in terms of the corresponding morphisms of stacks.  For example, a morphism of log schemes $(X, M_X)\rightarrow (S, M_S)$ is log smooth (resp.\ log flat) if and only if the corresponding morphism $X\rightarrow \logs _{(S, M_S)}$ is smooth (resp.\ flat) in the ordinary sense (see \cite[Theorem 4.6]{O}).  Furthermore, the sheaf of logarithmic differentials $\Omega ^1_{(X, M_X)/(S, M_S)}$ is canonically isomorphic to the sheaf $\Omega ^1_{X/\logs _{(S, M_S)}}$, and this canonical isomorphism is compatible with the universal derivations.

\subsection{} The stack $\logs _{(S, M_S)}$ is somewhat unwieldy, but there is a simpler description in terms of toric stacks.  

For a fine monoid $P$, the affine scheme
$$
\Sp (\mathbb{Z}[P])
$$
has a natural action of the torus $D(P^{\gp}):= \Sp (\mathbb{Z}[P^{\gp}])$, and we can consider the stack quotient
$$
\Theta _P:= [\Sp (\mathbb{Z}[P])/D(P^{\gp})].
$$
By \cite[Proposition~5.20]{O} this stack can be viewed as the stack whose fiber over a scheme $T$ is the groupoid of pairs $(M_T, \beta )$, where $M_T$ is a fine log structure on $T$ and $\beta \colon P\rightarrow \overline M_T$ is a morphism of sheaves of monoids which fppf locally on $T$ lifts to a chart.   The universal pair is obtained by descent from the natural log structure on $\Sp (\mathbb{Z}[P])$ induced by the map $P\rightarrow \mathbb{Z}[P]$.

\begin{example}\label{E:2.3b}
If $M_T$ is a fine log structure on a scheme $T$ such that $\overline M_{T, \bar t}$ is either $0$ or $\mathbb{N}$ for every geometric point $\bar t\rightarrow T$, then there exists a unique morphism $\mathbb{N}\rightarrow \overline M_T$ which \'etale locally lifts to a chart.  One concludes from this that there is an open immersion $[\mathbb{A}^1/\mathbb{G}_m]\subset \logs _{(\Sp (\mathbb{Z}), \mls O_{\Sp (\mathbb{Z})}^*)}$ classifying such log structures.  The stack $[\mathbb{A}^1/\mathbb{G}_m]$ can also be viewed as associating to a scheme $T$ the groupoid of pairs $(\mls L, \rho )$, where $\mls L$ is a line bundle on $T$ and $\rho \colon \mls L\rightarrow \mls O_T$ is a morphism of line bundles.  This connection is discussed in \cite[Complement 1]{MR1463703}.
\end{example} 

\subsection{}\label{P:2.4} More generally, there is a relative version.  Let $(S, M_S)$ be a fine log scheme, and let $\beta _S\colon Q\rightarrow M_S$ be a chart, with $Q$ a fine monoid. Also fix  a morphism of fine monoids $\theta \colon Q\rightarrow P$ such that the associated map of groups $Q^{\gp}\rightarrow P^{\gp}$ is injective.  Let $G$ denote the quotient group $P^{\gp}/Q^{\gp}$.  We can then consider the stack quotient $\Theta$ over $S$ given by
$$
\Theta:= \left[\Sp _S\left(\mls O_S[P]\otimes _{\mls O_S[Q]}\mls O_S\right)/D(G)\right].
$$
By \cite[5.20]{O} the stack $\Theta$ can be viewed as the stack whose fiber over an $S$-scheme $f\colon T\rightarrow S$ is the groupoid of triples $(M_T, f^b, \gamma )$, where $M_T$ is a fine log structure on $T$, $f^b\colon f^*M_S\rightarrow M_T$ is a morphism of log structures on $T$, and $\gamma \colon P\rightarrow \overline M_T$ is a morphism of sheaves of monoids which fppf locally on $T$ lifts to a chart and which is such that the diagram
$$
\xymatrix{
P\ar[r]^-{\gamma } & \overline M_T\\
Q\ar[u]_-{\theta }\ar[r]^-{\beta }& f^{-1}M_k\ar[u]^-{f^b}}
$$
commutes.

\subsection{}\label{P:2.5} There is a natural map $\Theta\rightarrow \logs _{(S, M_S)}$ which is shown in \cite[Corollary~5.25]{O} to be \'etale.  For many purposes one can work with the stack $\Theta$ instead of the stack $\logs _{(S, M_S)}$.  The stack is sometimes easier to work with; for example, it has affine diagonal.

\begin{example} Let $V$ be a discrete valuation ring, and let $\pi \in V$ be a uniformizer defining a map $\mathbb{N}\rightarrow V$ and a log structure $M_V$ on $\Sp (V)$.  Let $f\colon X\rightarrow \Sp (V)$ be a semistable scheme over $V$ such that the preimage of the closed point $s\in \Sp (V)$ is a divisor with simple normal crossings in $X$.  Let $I$ denote the set of irreducible components of the closed fiber of $X$, and let $M_X$ denote the log structure on $X$ defined by the closed fiber. Then there is a natural map $\mathbb{N}^I\rightarrow \overline M_X$ which locally lifts to a chart.  We have a log smooth morphism
$$
f\colon (X, M_X)\lra (\Sp (V), M_V)
$$
and a commutative diagram of sheaves of monoids
$$
\xymatrix{
\mathbb{N}^I\ar[r]& \overline M_X\\
\mathbb{N}\ar[u]_-{\Delta }\ar[r]& f^{-1}\overline M_V\ar[u]^-{f^b}\rlap{.}}
$$
The toric stack $\Theta$ associated to this situation then has the following explicit description.  Choose an ordering $\{1, \dots, r\}\simeq I$.  We then have
$$
\Theta\simeq \left[\Sp (V[x_1, \dots, x_r]/(x_1\cdots x_r = \pi ))/\mathbb{G}_m^{r-1}\right],
$$
where $\underline u= (u_1, \dots, u_{r-1})$ acts on $V[x_1, \dots, x_r]$ by $\underline u*x_i = u_ix_i$ if $1\leq i\leq r-1$ and
$$
\underline u*x_r = (u_1\cdots u_{r-1})^{-1}x_r.
$$
\end{example}

\section{Definition of Hochschild homology}

\subsection{}\label{P:3.1} Let $k$ be a ring, let $\mls S/k$ be an algebraic stack with quasi-compact diagonal, and let
$$
f\colon X\lra \mls S
$$
be a flat morphism from a quasi-compact algebraic space $X$.  The purpose of this section is to define a complex 
$$
\mcH _{X/\mls S}\in D_{\qcoh}(X, \mls O_X)
$$
in the derived category of complexes
of $\mls O_X$-modules on the \'etale site of $X$ with quasi-coherent cohomology sheaves generalizing the usual Hochschild complex for schemes.

\subsection{}\label{P:3.2} First consider the case when $\mls S = \Sp (R)$ is an affine scheme.  In this case one can define $\mcH _{X/R}$ using the method of \cite[Section~4]{WG}. Consider the complex of presheaves which to any affine \'etale $X$-scheme $\Sp (A)\rightarrow X$ associates the Hochschild complex 
$$
C_{A/R}^*
$$
of $A/R$ defined in \cite[Section~9.1.1]{W2}.  The formation of this complex is functorial in the affine $X$-scheme, and we let $\mc C ^*_{X/R}$ denote the associated complex of sheaves of $\mls O_X$-modules on the \'etale site of $X$.

The \emph{Hochschild complex}, denoted by $\mcH _{X/R}$, is the image of $\mc C^*_{X/R}$ in the derived category of $\mls O_X$-modules.  However, we find it convenient to distinguished between $\mc C^*_{X/R}$ and $\mcH _{X/R}$ and reflect this in the notation.

The complex $\mc C^*_{X/R}$ is functorial in the sense that if 
$$
\xymatrix{
X'\ar[r]^-g\ar[d]^-{f'}& X\ar[d]^-f\\
\Sp (R')\ar[r]& \Sp (R)}
$$
is a commutative diagram of algebraic spaces, then there is an induced map of complex of sheaves of $\mls O_{X'}$-modules
$$
g^*\mc C^*_{X/R}\lra \mc C_{X'/R'}^*.
$$

\subsection{}\label{P:3.3}  To define the Hochschild complex in the general case, let $f\colon X\rightarrow \mls S$ be a flat  morphism from a quasi-compact algebraic space $X$ to an algebraic stack $\mls S$ with quasi-compact diagonal.  Let $S_\bullet \rightarrow \mls S$ be a smooth hypercover with each $S_n = \Sp (R_n)$ affine, and let $X_\bullet $ be the fiber product $X\times _{\mls S}S_\bullet $.  For each~$n$, we then have a complex $\mc C_{X_n/R_n}^*$, and by the functoriality already discussed, we have for every morphism $h\colon [m]\rightarrow [n]$ with induced map $X(h)\colon X_n\rightarrow X_m$ a morphism of complexes
\begin{equation}\label{E:pullmap}
X(h)^*\mc C_{X_m/R_m}^*\lra \mc C_{X_n/R_n}^*.
\end{equation}
We therefore get a complex of sheaves of $\mls O_{X_\bullet }$-modules $\mc C_{X_\bullet /R_\bullet }$ in the topos of sheaves on $X_{\bullet , \et}$
defined in \cite[Section~12.4]{LMB}.    Let $\mcH _{X_\bullet /S_\bullet }$ denote the image of $\mc C_{X_\bullet /R_\bullet }$ in $D(X_{\bullet, \et}, \mls O_{X_\bullet })$.

\begin{prop}\label{P:descent} The complex $\mcH _{X_\bullet /S_\bullet }$ lies in $D^-_{\qcoh}(X_\bullet , \mls O_{X_\bullet })$ and therefore $($by \cite[\href{https://stacks.math.columbia.edu/tag/0DC1}{Tag 0DC1}]{stacks-project}$)$ is isomorphic to the pullback of a unique object $\mcH _{X/\mls S}\in D(X, \mls O_X)$.
\end{prop}

\begin{proof}
It suffices to verify that \eqref{E:pullmap} is a quasi-isomorphism for all morphisms in $\Delta $ and that $\mcH _{X_0/R_0}$ lies in $D^-_{\qcoh}(X_0, \mls O_{X_0})$.  That $\mcH _{X_0/R_0}$ lies in $D^-_{\qcoh}(X_0, \mls O_{X_0})$ follows from \cite[Corollary~0.4 and Example~4.3]{WG}.  The statement that \eqref{E:pullmap} is a quasi-isomorphism is immediate in the case when $X_m$ is affine, and the general case follows from this by sheafification.
\end{proof}

\subsection{}\label{P:3.5} The complex $\mcH _{X/\mls S}$ is independent of the choice of hypercover $S_\bullet \rightarrow \mls S$.  This can be seen as follows.  Let $HR_{\aff}(\mls S)$ denote the category whose objects are hypercovers $S_\bullet $ of $\mls S$ with each $S_n$ affine, and whose morphisms are morphisms of hypercovers up to simplicial homotopy over $\mls S$ as in \cite[Section~V.7.3.1.6]{SGA4}.  By \cite[Th\'eor\`eme~V.7.3.2(1)]{SGA4}, the category $HR_{\aff}(\mls S)$ is cofiltering.

For each object $S_\bullet \in HR_{\aff}(\mls S)$, let $\mcH _{X/\mls S, S_\bullet }\in D^-_{\qcoh}(X, \mls O_X)$ be the object provided by Proposition~\ref{P:descent}.  It follows from the construction that, for a morphism of hypercovers $u\colon S'_\bullet \rightarrow S_\bullet $, we get an induced map
$$
e_u\colon \mcH _{X/\mls S, S_\bullet }\lra \mcH _{X/\mls S, S'_\bullet }.
$$

\begin{prop}\label{P:3.6}
  \leavevmode
  \begin{enumerate}
    \item\label{P:3.6-1}
  If two morphisms $u_0, u_1\colon S'_\bullet \rightarrow S_\bullet $ are related by a simplicial homotopy, then $e_{u_0} = e_{u_1}$.

\item\label{P:3.6-2} For any morphism $u\colon S'_\bullet \rightarrow S_\bullet $ of affine hypercovers of $\mls S$, the induced morphism $e_u\colon \mcH _{X/\mls S, S_\bullet }\rightarrow \mcH _{X/\mls S, S'_\bullet }$ is an isomorphism.
  \end{enumerate}
\end{prop}

\begin{proof}
We begin by proving~\eqref{P:3.6-1}.

For a simplicial algebraic space $Y_\bullet $, let $Y_\bullet \times \Delta _1$ denote the simplicial algebraic space given by
$$
[n]\longmapsto Y_n\times \Hom_\Delta ([n], [1]),
$$
so that a simplicial homotopy between $u_0$ and $u_1$ is given by a morphism of simplicial algebraic spaces 
$$
H\colon S'_\bullet \times \Delta _1\lra S_\bullet
$$
such that $u_i$ is obtained by composing $H$ with the map $S'_\bullet \rightarrow S'_\bullet \times \Delta _1$ defined by the morphism $[0]\rightarrow [1]$ with image $i$.  Let
$$
H_X\colon X'_\bullet \times \Delta _1\lra X_\bullet
$$
be the morphism induced by $H$.  Just as in the definition of $\mc C^*_{X_\bullet /R_\bullet }$, we get a complex $\mc C^*_{X_\bullet ^\prime \times \Delta _1/S'_\bullet \times \Delta _1}$ which is canonically isomorphic to $p^*\mc C^*_{X'_\bullet /S'_\bullet }$, where $p\colon X'_\bullet \times \Delta _1\rightarrow X'_\bullet $ is the projection.  For $i=0,1$ let
$$
e_i\colon X'_\bullet \lra X'_\bullet \times \Delta _1
$$
be the map defined by the map $[0]\rightarrow [1]$ sending $0$ to $i$.  By the functoriality of the Hochschild complex, we then get morphisms
$$
\alpha \colon H_X^*\mc C^*_{X'_\bullet /S_\bullet }\lra \mc C^*_{X'_\bullet \times \Delta _1/S'_\bullet \times \Delta _1}
$$
and
$$
\beta _i\colon e_i^*\mc C^*_{X'_\bullet \times \Delta _1/S'_\bullet \times \Delta _1}\lra \mc C^*_{X'_\bullet /S'_\bullet }.
$$
If $v_i\colon X'_\bullet \rightarrow X_\bullet $ denotes the map induced by $u_i$, then the composition of $\alpha $ and $\beta _i$ induces  a morphism
$$
\rho _i\colon v_i^*\mc C^*_{X_\bullet /S_\bullet }\lra \mc C^*_{X'_\bullet /S'_\bullet },
$$
and we need to show that the two induced maps in $D^+_{\qcoh}(X, \mls O_X)$ are equal.

Let $X^\Delta _\et$ denote the topos of cosimplicial sheaves on $X_{\et}$.  We then get a commutative diagram of topoi
$$
\xymatrix{
X'_{\bullet, \et} \ar@<.5ex>[r]^-{e_0}\ar@<-.5ex>[r]_-{e_1}\ar[rd]_-{\pi _1}& (X'_\bullet \times \Delta _1)_{\et }\ar[r]^-{H_X}\ar[d]^-{\pi _2}& X_{\bullet, \et }\ar[ld]^-{\pi _3}\\
& X_{\et }^\Delta \rlap{.}& }
$$
Let $\mathbb{Z}^{\Delta _1}$ be the object of $X_\et ^\Delta $ given by
$$
[n]\longmapsto \prod _{[n]\lra [1]}\mathbb{Z}.
$$
We then have
$$
R\pi _{2*}p^*\mc C^*_{X'_\bullet /S'_\bullet }\simeq R\pi _{1*}\mc C_{X'_\bullet /S'_\bullet }^*\otimes _{\mathbb{Z}}\mathbb{Z}^{\Delta _1}
$$
in $D(X_\et ^\Delta , \mls O_{X^\Delta _\et })$.  We therefore get a commutative diagram
$$
\xymatrix{
R\pi _{3*}\mc C_{X_\bullet /S_\bullet }\ar[r]^-{\bar \alpha }& R\pi _{1*}\mc C_{X'_\bullet /S'_\bullet }\otimes _{\mathbb{Z}}\mathbb{Z}^{\Delta _1}\ar@<.5ex>[r]^-{\bar \beta _1}\ar@<-.5ex>[r]_-{\bar \beta _2}& R\pi _{1*}\mc C_{X'_\bullet /S'_\bullet },}
$$
where we write $\bar \alpha $ (resp.\ $\bar \beta _i$) for the map induced by $\alpha $ (resp.\ $\beta _i$).  To prove~\eqref{P:3.6-1} it suffices to show that the two compositions $\bar \beta _i\circ \bar \alpha $ induce the same map in the derived category after passing to total complexes.  This follows from the general fact that the two maps induced by the $\bar \beta _i$ are equal.

In light of~\eqref{P:3.6-1}, to prove~\eqref{P:3.6-2} it suffices to consider the case when all of the morphisms $u_n\colon S'_n\rightarrow S_n$ are flat.  In this case it is immediate from the construction that if $u_X\colon X'_\bullet \rightarrow X_\bullet $ is the induced morphism of simplicial algebraic spaces, then the natural map
$$
u_X^*\colon \mc C^*_{X_\bullet /S_\bullet }\lra \mc C^*_{X'_\bullet /S_\bullet ^\prime }
$$
is a quasi-isomorphism.  Statement~\eqref{P:3.6-2} follows.
\end{proof}

\subsection{}\label{P:3.7} It follows that we get a functor
$$
\mcH _{X/\mls S, -}\colon HR_{\aff}(\mls S)^{\op}\lra D^-_{\qcoh}(X, \mls O_X)
$$
such that the colimit
$$
\mcH _{X/\mls S}:= \colim _{S_\bullet \in HR_{\aff}(\mls S)}\mcH _{X/\mls S, S_\bullet }
$$
exists, and for every $S_\bullet $ the map
$$
\mcH _{X/\mls S, S_\bullet }\lra \mcH _{X/\mls S}
$$
is an isomorphism.    This therefore gives a definition of $\mcH _{X/\mls S}$ independent of the choice of hypercover.

\begin{rem} Using the site $\mls C_{X/\mls S}$ introduced in Section~\ref{S:section5.5}, one can give an alternate construction of $\mcH _{X/\mls S}$ which does not require the use of hypercovers.
\end{rem}

\begin{defn}\label{D:3.8}
The \emph{Hochschild homology} of $X/\mls S$ is defined to be the $k$-modules 
$$
HH_n(X/\mls S):= H^{-n}(X, \mcH _{X/\mls S}).
$$
The \emph{Hochschild cohomology} of $X/\mls S$ is defined as 
$$
HH^n(X/\mls S):= \Ext^n_X(\mcH _{X/\mls S}, \mls O_X).
$$

Similarly, we define the \emph{Hochschild homology sheaves} of $X/\mls S$ to be 
$$
\mcH _{X/\mls S, n}:= \mls H^{-n}(\mcH _{X/\mls S}).
$$
\end{defn}

\begin{rem} This also gives a definition of Hochschild homology for a morphism of log schemes $f\colon (X, M_X)\rightarrow (S, M_S)$, with $X$ quasi-compact.  We define $\mcH _{(X, M_X)/(S, M_S)}\in D^-(X, \mls O_X)$ to be the complex $\mcH _{X/\logs _{(S, M_S)}}$ and the Hochschild homology groups as 
$$
HH_n((X, M_X)/(S, M_S)):= HH_n(X/\logs _{(S, M_S)}).
$$
\end{rem}

\section{Relation with the diagonal}

We continue with the notation of Section~\ref{P:3.3}.

\subsection{}\label{P:4.1}
Consider the fiber product $I:= X\times _{\mls S}X$ and the diagonal morphism $\Delta _{X/\mls S}\colon X\rightarrow I$.  The fiber product $I$ is an algebraic space since $\mls S$ is algebraic, but in general it is not separated and the diagonal $\Delta _{X/\mls S}$ is not a closed immersion (for example this is the case when $\mls S = \logs _{(k, M_k)}$).  Nonetheless, we can consider the sheaf on the small \'etale site $\Delta _X^{-1}\mls O_I$.  This is an $\mls O_X-\mls O_X$ bialgebra equipped with an epimorphism $q\colon \Delta _X^{-1}\mls O_I\rightarrow \mls O_X$.

\begin{thm}\label{T:4.2} There is a natural isomorphism
$$
\mls O_X\otimes ^{\mathbf{L}}_{\Delta _X^{-1}\mls O_I}\mls O_X\simeq \mcH _{X/\mls S}
$$
in $D(\mls O_X)$.
\end{thm}

\begin{proof}
To ease notation write $P_{X/\mls S}$ for $\mls O_X\otimes ^{\mathbf{L}}_{\Delta _X^{-1}\mls O_I}\mls O_X$, viewed as an object of $D(\mls O_X)$ through the action of~$\mls O_X$ on the left factor.

    Consider a smooth affine hypercover $S_\bullet := \Sp (R_\bullet )\rightarrow \mls S$, and let $X'_\bullet := \Sp (B_\bullet )\rightarrow X_{S_\bullet }$ be an \'etale  refinement of the hypercover $q\colon X_{S_\bullet }\rightarrow X$ by affines, so we have a commutative diagram
    $$
    \xymatrix{
    \Sp (B_\bullet )\ar[r]\ar[d]& X\ar[d]\\
    \Sp (R_\bullet )\ar[r]& \mls S\rlap{,}}
    $$
    where the horizontal arrows are smooth hypercovers and all of the maps $\Sp (B_n)\rightarrow \Sp (R_n)\times _{\mls S}X$ are \'etale.

    We then obtain a commutative diagram
    $$
    \xymatrix{
    X_\bullet '\ar[d]\ar[r]^-q& X_{S_\bullet }\ar[d]\ar[r]^-{\pi }& X\ar[d]\\
    X'_\bullet \times _{S_\bullet }X'_\bullet \ar[r]^-Q& X_{S_\bullet }\times _{S_\bullet }X_{S_\bullet }\ar[r]\ar[d]^-{p_1}& X\times _{\mls S}X\ar[d]^-{p_1}\\
    & X_{S_\bullet }\ar[r]\ar[d]& X\ar[d]\\
    & S_\bullet \ar[r]& \mls S\rlap{,}}
    $$
    where the squares in the right column are Cartesian and the  maps $q$ and $Q$ are \'etale.  It follows from this that the natural maps
    $$
    \pi ^*P_{X/\mls S}\lra P_{X_{S_\bullet }/S_\bullet }, \ \ q^*P_{X_{S_\bullet }/S_\bullet }\lra P_{X'_\bullet /S_\bullet }
    $$
    are isomorphisms. 

    
    In particular, by cohomological descent it suffices to construct an isomorphism
    $$
    P_{X'_\bullet /S_\bullet }\simeq \mcH _{X_\bullet /S_\bullet }|_{X'_\bullet }.
    $$
This we can do using the usual comparison as in \cite[Lemma~9.1.3]{W2}.  Namely, for each $n$ we can consider the (unnormalized) bar resolution $\mc B_{R_n}(B_n, B_n)$, which is a projective resolution of $B_n$ as a $(B_n, B_n)$-bimodule and for which $\mc B_{R_n}(B_n, B_n)\otimes _{B_n\otimes _{R_n}B_n}B_n=C^*_{B_n/R_n}$.  The complex $\mc B_{R_n}(B_n,B_n)$ is functorial in $n$, as is the identification with $C^*_{B_n/R_n}$.  It follows that we get an explicit complex $\mc B_{R_\bullet }(B_\bullet , B_\bullet )\otimes _{B_\bullet \otimes B_\bullet }B_\bullet $ on $X'_{\bullet }$ representing $P_{X'_\bullet /S_\bullet }$, and this complex is identified with $C^*_{B_\bullet /R_\bullet }$.
\end{proof}

The following corollary, together with Section~\ref{P:2.5}, often lets one describe log Hochschild cohomology more concretely using toric stacks.

\begin{cor} Let $\mls S'\rightarrow \mls S$ be an \'etale morphism of algebraic stacks with quasi-compact diagonal, and suppose $X\rightarrow \mls S'$ is a flat morphism from a quasi-compact algebraic space.  Then the natural map
$$
\mcH _{X/\mls S}\lra \mcH_{X/\mls S'}
$$
is an isomorphism.
\end{cor}
\begin{proof}
    Consider the commutative diagram
    $$
    \xymatrix{
    X\ar@/^2pc/[rr]^-{\id}\ar[rd]_{\Delta '}\ar[r]^-s& W\ar[d]^-\delta \ar[r]^-q& X\ar[d]^-\Delta \\
    & X\times _{\mls S'}X\ar[r]& X\times _{\mls S}X\rlap{,}}
    $$
    where $W$ is defined to make the square Cartesian.  Then $q$ is \'etale and $s$ is a section of $q$.  It follows that $s$ is an open immersion.  Therefore, 
    $$
    q^{-1}(\mls O_X\otimes _{\Delta ^{-1}\mls O_{X\times _{\mls S}X}}^{\mathbf{L}}\mls O_X)\simeq \mls O_W\otimes ^{\mathbf{L}}_{\delta ^{-1}\mls O_{X\times _{\mls S'}X}}\mls O_W,
    $$
    and, restricting these along the open immersion $s$, we get $\mls O_X\otimes ^{\mathbf{L}}_{\Delta ^{\prime -1}\mls O_{X\times _{\mls S'}X}}\mls O_X$. From this and Theorem~\ref{T:4.2}, the result follows.
\end{proof}

\begin{rem} In the case of a morphism of fine log schemes $(X, M_X)\rightarrow (S, M_S)$ with associated map $X\rightarrow \logs _{(S, M_S)}$, the resulting diagonal map $X\rightarrow X\times _{\logs _{(S, M_S)}}X$ is the exactification of the diagonal as discussed in \cite[Appendix A]{OlssonTowards}.  In the situation of Section~\ref{P:2.4}, when one has a chart, one can also consider the diagonal map $X\rightarrow X\times _{\Theta }X$.  Since $\Theta \rightarrow \logs _{(S, M_S)}$ is \'etale by Section~\ref{P:2.5}, the natural map $X\times _{\Theta }X\rightarrow X\times _{\logs _{(S, M_S)}}X$ is also \'etale, and for most purposes one can work with  the diagonal either over $\Theta $ or over $\logs _{(S, M_S)}$.  The fiber product $X\times _{\Theta }X$ is the log diagonal defined using a frame in \cite{MR2102698}.
\end{rem}

\begin{example} Let $X/k$ be a smooth scheme with a smooth divisor $D\subset X$.  Let $M_X$ denote the log structure associated to $D$ and the associated map $M_X\colon X\rightarrow \logs $ to the stack of log structures over $k$.  As discussed in Example~\ref{E:2.3b}, this map factors through $[\mathbb{A}^1/\mathbb{G}_m]\subset \logs $, and the corresponding line bundle with map to $\mls O_X$ is simply $\mls O_X(-D)\hookrightarrow \mls O_X$.  The fiber product $X\times _{\logs }X\simeq  X\times _{[\mathbb{A}^1/\mathbb{G}_m]}X$ is therefore the functor which to a scheme $(f_1, f_2)\colon T\rightarrow X\times X$ associates the set of isomorphisms $\sigma \colon f_1^*\mls O_X(-D)\rightarrow f_2^*\mls O_X(-D)$ compatible with the given maps to $\mls O_T$.  This can be described more classically as the complement in the blowup $B\rightarrow X\times X$ of $D\times D\subset X\times X$ of the strict transforms of the divisors $D\times X$ and $X\times D$.  If $\delta \colon X\hookrightarrow B$ is the unique lift of the diagonal $\Delta \colon X\hookrightarrow X\times X$, then it follows from a calculation in local coordinates that the induced map $\Omega ^1_X = \Delta ^*I_\Delta \rightarrow \delta ^*I_\delta $ between the pullbacks of the ideal sheaves identifies $\delta ^*I_\delta $ with $\Omega ^1_{(X, M_X)}$, and the log Hochschild homology of $(X, M_X)$ is locally the exterior algebra on $\Omega ^1_{(X, M_X)}$.  Furthermore, the natural map from the ordinary Hochschild homology of $X$ to the log Hochschild homology is locally described by the map $\Omega ^1_X\rightarrow \Omega ^1_{(X, M_X)}$.
\end{example}

\section{Alternate description in the case of affine diagonal}\label{S:section4}

\subsection{}\label{P:1.1} Let $k$ be a commutative ring with unit, let $\mls S$ be an algebraic stack over $k$, and assume the diagonal
$$
\Delta \colon \mls S\lra \mls S\times _{\Sp (k)}\mls S
$$
is affine.  Let $f\colon X\rightarrow \mls S$ be a morphism from a quasi-compact algebraic space $X$ to $\mls S$.  In this case we can define a complex of $\mls O_X$-modules $\mc C^*_{X/\mls S}$ representing the object $\mcH _{X/\mls S}\in D(X, \mls O_X)$, generalizing the definition of $\mc C^*_{X/R}$ for a scheme over a ring $R$.

\subsection{} For $n\geq 0$ let $X_{[n]}$ denote the $(n+1)$-fold fiber product
$$
 \overbrace{X\times _{\mls S}X\cdots X\times _{\mls S}X}^{n+1}.
$$
If $X= \Sp (A)$ is affine, then, since the diagonal of $\mls S$ is affine, the scheme $X_{[n]}$ is also affine and we let $C^n_{A/\mls S}$ denote $\Gamma (X_{[n]}, \mls O_{X_{[n]}})$ viewed as an $A$-module via the first projection.

For $i=0, \dots, n-1$ let 
$$
d_i\colon X_{[n-1]}\lra X_{[n]}
$$
denote the map induced by the diagonal $X\rightarrow X\times _{\mls S}X$ in the $\supth{i}$ position and the identity on the other factors, and let $d_n\colon X_{[n-1]}\rightarrow X_{[n]}$ denote the map given on scheme-valued points by
$$
(x_0, \dots, x_{n-1})\longmapsto (x_0, x_1, \dots, x_{n-1}, x_0).
$$
Also, for $i=1, \dots, n$ let 
$$
s_i\colon X_{[n]}\lra X_{[n-1]}
$$
denote the projecting forgetting the $\supth{i}$ coordinate.

In the affine case these maps give the $A$-modules  $C^n_{A/\mls S}$ the structure of a simplicial $A$-module.  We let $C^*_{A/\mls S}$ denote the associated complex of $A$-modules.  In this way we get a complex of presheaves 
$$
A\longmapsto C^*_{A/\mls S}, 
$$
and we let $\mc C^*_{X/\mls S}$ denote the associated complex of sheaves of $\mls O_X$-modules.

The formation of the complex $\mc C^*_{X/\mls S}$ is functorial in the morphism $X\rightarrow \mls S$, and in the case when $\mls S$ is a scheme, the complex $\mc C^*_{X/\mls S}$ agrees with that defined in Section~\ref{P:3.3}.

\begin{rem} As in the case of schemes, the terms of the complex $\mc C^*_{X/\mls S}$ are not quasi-coherent.  For example, if $X = \Sp (A)\rightarrow \mls S = \Sp (B)$ is a morphism of affine schemes and $n\geq 1$, then $C^n_{A/B} = A\otimes _B\cdots \otimes _BA$ (the $(n+1)$-fold tensor product).  For an element $f\in A$ with localization $A_f$, we then have $C^n_{A/B}\otimes _AA_f\simeq A_f\otimes _BA\cdots \otimes _BA$ ($A_f$ in the first factor and $A$ in the other factors), which typically is not equal to the $(n+1)$-fold tensor product of $A_f$ over $B$.
\end{rem}

\subsection{} 
If $S_\bullet \rightarrow \mls S$ is a smooth hypercover with each $S_n = \Sp (R_n)$ affine, and if $\pi _\bullet \colon X_\bullet \rightarrow X$ is the associated hypercover of $X$, we  get by functoriality a morphism
\begin{equation}\label{E:4.4.1}
\pi _\bullet ^{-1}\mc C^*_{X/\mls S}\lra \mc C^*_{X_\bullet /R_\bullet }.
\end{equation}

\begin{prop} The induced map 
$$
\mc C^*_{X/\mls S}\lra \mcH _{X/\mls S} = R\pi _*\mc C^*_{X_\bullet /R_\bullet }
$$
in $D(X, \mls O_X)$ is an isomorphism.
\end{prop}
\begin{proof}
It suffices to show that if $X = \Sp (A)$ is an affine scheme, then the map
$$
C^*_{A/\mls S}\lra R\Gamma (X, \mcH _{X/\mls S})
$$
is an isomorphism in the derived category of $A$-modules.   

Since the diagonal of $\mls S$ is affine, each $X_n$ is affine, so $X_\bullet $ is equal to the simplicial scheme obtained by applying $\Sp $ to a cosimplicial $A$-algebra $A_\bullet $.  Let $\Mod_{A_\bullet }^\Delta $ denote the category of modules over $A_\bullet $, so an object $M_\bullet \in \Mod_{A_\bullet }^\Delta $ is a cosimplicial $A$-module equipped with an action of $A_\bullet $.
We then have  a commutative diagram of derived categories
$$
\xymatrix{
D\left(X_\bullet , \mls O_{X_\bullet }\right)\ar[r]^-{Re_{\bullet *} }\ar[d]^-{R\pi _*}& D\left(\Mod_{A_\bullet }^\Delta\right)\ar[d]^-\int \\
D(X, \mls O_X)\ar[r]^-{R\Gamma }& D(\Mod_A)\rlap{,}}
$$
where
$$
\int \colon D\left(\Mod_{A_\bullet }^\Delta \right)\lra D^+(\Mod_A)
$$
is the functor taking the total complex and $Re_\bullet $ is the derived functor of the level-wise global section functor.

By \cite[Corollary~0.4]{WG} the natural map in $D(\Mod_{A_\bullet }^\Delta )$
$$
C_{A_\bullet /R_\bullet }^*\lra Re_{\bullet *} \mcH _{X_\bullet /R_\bullet }
$$
is an isomorphism.  It therefore suffices to show that the natural map
$$
C_{A/\mls S}^*\lra \int C^*_{A_\bullet /R_\bullet }
$$
is an isomorphism in $D(\Mod_{A})$.  For this in turn it suffices to show that each of the maps
$$
C_{A/\mls S}^n\lra \int C^n_{A_\bullet /R_\bullet }
$$
is a quasi-isomorphism.  This follows from observing that $C^n_{A_\bullet /R_\bullet }$ is isomorphic to the complex obtained by taking global sections of the simplicial affine scheme given by the fiber product
\begin{equation*}\pushQED{\qed}
X_{[n]}\times _{\mls S}S_\bullet .
\qedhere \popQED
	\end{equation*}
\renewcommand{\qed}{}     
\end{proof}

\section{Interlude on differentials}\label{S:section5.5}

\subsection{} If $X\rightarrow S$ is a morphism of schemes and $E$ is a quasi-coherent sheaf on $X$, then the universal derivation $d\colon \mls O_X\rightarrow \Omega ^1_{X/S}$ establishes a bijection between the $\mls O_X$-module of derivations $\partial \colon \mls O_X\rightarrow E$ and $\mls O_X$-linear maps $\Omega ^1_X\rightarrow E$.

For a morphism $X\rightarrow \mls S$ from a scheme to an algebraic stack, however, the $\mls O_X$-module of maps $\Omega ^1_{X/\mls S}\rightarrow E$ can no longer be described easily using maps $\mls O_X\rightarrow E$ (at least if we only consider the small Zariski or \'etale topology).

\begin{example}\label{E:5.2}
Let $k$ be a field, and consider the map $\Sp (k)\rightarrow B\mathbb{G}_{m, k}$ corresponding to the trivial torsor.  In this case the $k$-space  $\Omega ^1_{\Sp (k)/B\mathbb{G}_m}$ is $1$-dimensional, and the derivation $d\colon k\rightarrow \Omega ^1_{\Sp (k)/B\mathbb{G}_m}$ is the zero map.  For a $k$-vector space $E$, the space $\Hom_k(\Omega ^1_{\Sp (k)/B\mathbb{G}_m}, E)$ is therefore isomorphic to $E$, whereas 
the space of $k$-linear derivations $k\rightarrow E$ is  $0$.
\end{example}

\subsection{} We can address this issue by considering a bigger topology, which is intermediate between the big and small \'etale site, as follows.  We call this topology the \emph{relative \'etale topology}.

Let $f\colon X\rightarrow \mls S$ be a morphism from an algebraic space to an algebraic stack.  Let $\mls C_{X/\mls S}$ be the category of $2$-commutative diagrams
\begin{equation}\label{E:5.1}
\xymatrix{
U\ar[d]\ar[r]& X\ar[d]\\
S\ar[r]& \mls S,}
\end{equation}
where $S\rightarrow \mls S$ is a smooth morphism from  a scheme and the induced map $U\rightarrow X_S:= X\times _{\mls S}S$ is \'etale.   We usually abbreviate the data of such an object by $(S, U\rightarrow X_S)$. A morphism
$$
(S', U'\rightarrow X_{S'})\lra (S, U\rightarrow X_S)
$$
in $\mls C_{X/\mls S}$ is given by a $2$-commutative diagram
$$
\xymatrix{
& U\ar[dd]\ar[rd]& \\
U'\ar[dd]\ar[rr]\ar[ru]&& X\ar[dd]\\
& S\ar[rd]& \\
S'\ar[rr]\ar[ru]&& \mls S\rlap{.}}
$$

We call a collection of morphisms
$$
\{(S_i, U_i\rightarrow X_{S_i})\lra (S, U\rightarrow X_S)\}
$$
a covering if the collections $\{S_i\rightarrow S\}$ and $\{U_i\rightarrow U\}$ are \'etale coverings.  We let $\Cov(S, U\rightarrow X_S)$ denote the set of coverings of $(S, U\rightarrow X_S)$.

\begin{lem} The category $\mls C_{X/\mls S}$ together with the coverings described above form a site $($in the sense of\, \cite[\href{https://stacks.math.columbia.edu/tag/03NF}{Tag 03NF}]{stacks-project}$)$.
\end{lem}

\begin{proof}
The axioms (1) and (2) in \cite[\href{https://stacks.math.columbia.edu/tag/03NF}{Tag 03NF}]{stacks-project} are immediate from the corresponding properties of the \'etale site of an algebraic space. To verify axiom (3), let $\{(S_i, U_i\rightarrow X_{S_i})\rightarrow (S, U\rightarrow X_S)\}$ be a covering, and let $(T, V\rightarrow X_T)\rightarrow (S, U\rightarrow X_S)$ be an arbitrary morphism.  Note that $X_{S_i}\times _{X_S}X_T\simeq X_{S_i\times _ST}$, so taking fiber products, we see that the morphism $U_i\times _UV\rightarrow X_{S_i}\times _{X_S}X_T\simeq X_{S_i\times _ST}$ is \'etale.  Therefore, 
$$
\left(S_i\times _ST, U_i\times _UV\rightarrow X_{S_i\times _ST}\right)
$$
is an object of $\mls C_{X/\mls S}$, which one verifies is the fiber product of $(S_i, U_i\rightarrow X_{S_i})$ and $(T, V\rightarrow X_T)$ over $(S, U\rightarrow X_S)$.  Furthermore, the collection
$$
\left\{\left(S_i\times _ST, U_i\times _UV\rightarrow X_{S_i\times _ST}\right)\lra (T, V\rightarrow X_T)\right\}
$$
is a covering, which verifies axiom (3) in \cite[\href{https://stacks.math.columbia.edu/tag/03NF}{Tag 03NF}]{stacks-project}.
\end{proof}

\subsection{}
For any smooth $S\rightarrow \mls S$ we have a functor
$$
\Et(X_S)\lra \mls C_{X/\mls S}, \quad (U\rightarrow X_S)\longmapsto (S, U\rightarrow X_S).
$$
This functor evidently takes coverings to coverings, and therefore if $F$ is a sheaf on $\mls C_{X/\mls S}$, we get by restriction a sheaf $F_S$ on the \'etale site of $X_S$.

\subsection{}
Let $\mls C^{\sim }_{X/\mls S}$ denote the topos of sheaves on $\mls C_{X/\mls S}$.  Let $\mls O_{X/\mls S}$ denote the sheaf which to any object $(S, U\rightarrow X_S)$ associates $\Gamma (U, \mls O_U)$ (the sheaf property follows immediately from the corresponding sheaf property for the \'etale site).  An $\mls O_{X/\mls S}$-module $E$ on $\mls C_{X/\mls S}$ is called \emph{quasi-coherent} if the restriction $E_S$ to $X_S$ is quasi-coherent for all smooth $S\rightarrow X_S$ and if, for any morphism $u\colon S'\rightarrow S$ of smooth $\mls S$-schemes, the induced map $u^*E_S\rightarrow E_{S'}$ of quasi-coherent sheaves on $X_{S'}$ is an isomorphism.

Let $\Qcoh(\mls C_{X/\mls S})$ denote the category of quasi-coherent sheaves on $\mls C_{X/\mls S}$, and let $\Qcoh(X)$ denote the category of quasi-coherent sheaves on $X$. There is a functor
\begin{equation}\label{E:5.6.1}
h\colon \Qcoh(X)\lra \Qcoh(\mls C_{X/\mls S})
\end{equation}
sending a quasi-coherent sheaf $E$ on $X$ to the sheaf on $\mls C_{X/\mls S}$ associating to $(S, U\rightarrow X_S)$ the sections $\Gamma (U, E_S)$, where $E_S$ denotes the pullback of $E$ to $X_S$.  To ease notation, if $E$ is a quasi-coherent sheaf on $X$, we sometimes write $E_{X/\mls S}$ for $h(E)$.

\begin{lem}\label{L:5.7.1} The functor \eqref{E:5.6.1} is an equivalence of categories.
\end{lem}
\begin{proof}
Fix a smooth cover $S\rightarrow \mls S$ with $S$ an algebraic space, and let $S_\bullet \rightarrow \mls S$ be the associated simplicial presentation.  Let $X_\bullet \rightarrow X$ be the simplicial presentation $X\times _{\mls S}S_\bullet $ of $X$ obtained by base change.   We then have a simplicial object in $\mls C_{X/\mls S}$ given by
$$
[n]\longmapsto \left(S_n, \id\colon X_{S_n}\rightarrow X_{S_n}\right).
$$
Therefore, if $E$ is a quasi-coherent sheaf on $\mls C_{X/\mls S}$, we obtain a quasi-coherent sheaf, in the sense of \cite[Proposition~13.2.4]{LMB}, on the simplicial scheme $X_\bullet $ given by $E_{S_n}$ on $X_n = X_{S_n}$.  By \textit{loc.~cit.}~this quasi-coherent sheaf is induced by a unique quasi-coherent sheaf on $X$.  It follows that we get a functor
$$
j\colon \Qcoh(\mls C_{X/\mls S})\lra \Qcoh(X).
$$
Furthermore, it is immediate from the construction that $j\circ h$ is isomorphic to the identity functor.  The functor $j$ can be described without using a covering of $\mls S$ as follows.  For an \'etale $X$-scheme $v\colon V\rightarrow X$, we have a functor
$$
v\colon \mls C_{V/\mls S}\lra \mls C_{X/\mls S}, \quad (S,U\rightarrow V_S)\longmapsto (S, U\rightarrow V_S\rightarrow X_S).
$$
Composition with $v$ defines a functor
$$
v^*\colon  \mls C^{\sim }_{X/\mls S}\lra \mls C^{\sim }_{V/\mls S}.
$$
By the construction of $j$, for a quasi-coherent sheaf $E$ on $\mls C_{X/\mls S}$, we have 
$$
j(E)(v\colon V\rightarrow X) = \Gamma (\mls C_{V/\mls S}, v^*E)
$$
functorially in $V$.  From this it follows in particular that there is a natural map $hj(E)\rightarrow E$ which is an isomorphism by descent theory.
\end{proof}

\subsection{} There is also a functor
\begin{equation}\label{E:6.8.1}
\Liset (\mls S)\lra \mls C_{X/\mls S}, \quad S\longmapsto (S, \id\colon X_S\rightarrow X_S).
\end{equation}
This functor induces a morphism of topoi
$$
\pi \colon \mls C^{\sim }_{X/\mls S}\lra \mls S_{\liset}.
$$
Note that by general considerations the functor \eqref{E:6.8.1} induces a pair of adjoint functors $(\pi ^{-1}, \pi _*)$, but this is not automatically a morphism of topoi since $\Liset (\mls S)$ does not have finite projective limits (see \cite{Olssonsheaves}).  In this case, however, we can describe the functor $\pi ^{-1}$ as follows, showing directly that $\pi ^{-1}$ does in fact commute with finite projective limits.

Recall from \cite[Lemme~12.2.1]{LMB} that the category of sheaves on $\mls S_{\liset }$ can be described as the category of systems $(F_U, \theta _\varphi )$ consisting of a sheaf $F_U$ on $U_\et $ for every smooth morphism $U\rightarrow \mls S$ and a map $\theta _{\varphi }\colon \varphi ^{-1}F_U\rightarrow F_V$ for every $\mls S$-morphism $\varphi \colon V\rightarrow U$ of schemes smooth over $\mls S$.  The morphisms $\theta _{\varphi }$ are further required to satisfy the condition that $\theta _{\varphi }$ is an isomorphism if $\varphi $ is \'etale and a cocycle condition.

With this description of $\mls S_\liset $, the functor $\pi _*$ sends a sheaf $E$ on $\mls C_{X/\mls S}$ to the system of sheaves given by associating to $U\rightarrow \mls S$ the pushforward $f_{U*}E_U$  with the natural transition maps, where $f_U\colon X_U\rightarrow U$ is the base change of $U$, $E_U$ is the restriction of $E$ to the \'etale site of $X_U:= X\times _{\mls S}U$,  and we consider pushforward with respect to the \'etale topology.  

 From this it follows that $\pi ^{-1}$ sends a sheaf $F$ on $\mls S_\liset $ given by a system $(F_U, \theta _\varphi )$ to the sheaf on $\mls C_{X/\mls S}$ 
given by the collection of sheaves $f_U^{-1}F_U$.  In particular, since $f_U^{-1}$ commutes with finite projective limits, it follows that $\pi ^{-1}$ does as well.  Furthermore, from this description it also follows that  $\pi $ extends naturally to a morphism of ringed topoi, denoted by the same letter,
$$
\pi \colon (\mls C^{\sim }_{X/\mls S}, \mls O_{X/\mls S})\lra \left(\mls S_{\liset }, \mls O_{\mls S_{\liset }}\right).
$$

\subsection{}
The reason for introducing the site $\mls C_{X/\mls S}$ is that this enables us to describe the universal property of the relative de Rham complex $\Omega ^\bullet _{X/\mls S}$.

\begin{defn} A \emph{quasi-coherent sheaf of commutative differential graded algebras} is a quasi-coherent sheaf of graded commutative $\mls O_X$-algebras $A ^\bullet = \oplus _{p\geq 0}A^p$ equipped with a $\pi ^{-1}\mls O_{\mls S_\liset }$-linear map
$$
d\colon A_{X/\mls S}^\bullet \lra A_{X/\mls S}^\bullet 
$$
of sheaves on $\mls C_{X/\mls S}$ such that $d^2 = 0$ and that, for any two local sections $a\in A^p$ and $b\in A^q$, we have 
$$
d(a\cdot b) = (da)\cdot b+(-1)^padb.
$$
Morphisms of quasi-coherent sheaves of commutative differential graded algebras are defined to be morphisms of $\mls O_X$-algebras compatible with the differentials $d$.  We denote the category of quasi-coherent sheaves of commutative differential algebras by $\dga_{X/\mls S}$.
\end{defn}

\begin{rem} Note that the differential $d$ induces a map $A^\bullet \rightarrow A^\bullet $ on $X_{\et}$, but $d$ cannot be recovered from this map, as Example~\ref{E:5.2} illustrates.  Indeed, in this example we have $A^\bullet = k\oplus \Omega ^1_{k/B\mathbb{G}_m}$, and $\Omega ^1_{k/B\mathbb{G}_m}$ is the $1$-dimensional vector space with basis $\frac{dt}{t}$, where $t$ is a coordinate on $\mathbb{G}_m$.  This identification is obtained by considering the Cartesian square
$$
\xymatrix{
\Sp (k)\ar[d]& \mathbb{G}_m\ar[l]_-q\ar[d]^-q\\
B\mathbb{G}_m& \Sp (k)\ar[l]}
$$
and noting that the pullback map $q^*\colon \Omega ^1_{k/B\mathbb{G}_m}\rightarrow \Omega ^1_{\mathbb{G}_m/k}$ identifies $\Omega ^1_{k/B\mathbb{G}_m}$ with the $\mathbb{G}_m$-invariants in $\Omega ^1_{\mathbb{G}_m/k}$.  The map $k[t^\pm ]\rightarrow \Omega ^1_{\mathbb{G}_m/k} = q^*\Omega ^1_{k/B\mathbb{G}_m}$ is the usual differential; in particular, it is nonzero. Therefore, the map $d$ on $\mls C_{k/B\mathbb{G}_m}$ is nonzero, but the restriction of $d$ to the \'etale site of $\Sp (k)$ is $0$.
\end{rem}

\subsection{}
The de Rham complex $\Omega ^\bullet _{X/\mls S}$ is naturally viewed as an object of $\dga_{X/\mls S}$.  To avoid confusion of notation, let us write simply $\Omega ^1$ for the quasi-coherent sheaf on $X_{\et}$ of relative differentials of the morphism $X\rightarrow \mls S$, and $\Omega ^1_{X/\mls S}$ for the corresponding quasi-coherent sheaf on $\mls C_{X/\mls S}$. Similarly, we write $\Omega ^p$ for the $\supth{p}$ exterior power of $\Omega ^1$ and $\Omega ^p_{X/\mls S}$ for the corresponding sheaf on $\mls C_{X/\mls S}$ (which is also the $\supth{p}$ exterior power of $\Omega ^1_{X/\mls S}$).  Note that $\Omega ^p_{X/\mls S}$ is the sheaf on $\mls C_{X/\mls S}$ which associates to any $(S, U\rightarrow X_S)$ the module $\Gamma (U, \Omega ^p_{U/S})$.   The additional structure of the differential is simply obtained by taking the differential $d\colon \oplus _p\Omega ^p_{U/S}\rightarrow \oplus _p\Omega ^p_{U/S}$ for each $(S, U\rightarrow X_S)$.  

Note that the differential graded algebra $\Omega ^\bullet _{X/\mls S}$ can also be described as the de Rham complex of the morphism of sheaves of rings $\pi ^{-1}\mls O_{\mls S_\liset }\rightarrow \mls O_{X/S}$ on $\mls C_{X/\mls S}$.  From this and the universal property of the usual de Rham complex, we obtain the following. 

\begin{prop} The de Rham complex $\Omega ^\bullet _{X/\mls S}\in \dga_{X/\mls S}$ is an initial object in the category $\dga_{X/\mls S}$.
\end{prop}
\begin{proof} This follows from the preceding discussion.
\end{proof}

\begin{rem}
There is an alternate description of the universal property of $\Omega ^1_{X/\mls S}$ which can be described as follows (see also \cite[Section~17.11 to the end of Chapter~17]{LMB}).  We will not use this in what follows.

Let $E$ be a quasi-coherent sheaf on $X$, and let $\mls O_X[E]$ be the sheaf of algebras whose underlying sheaf of $\mls O_X$-modules is $\mls O_X\oplus E$ and algebra structure is given on local sections by
$$
(f, e)\cdot (f', e') = (ff', fe'+f'e).
$$
Define $X[E]$ to be the relative spectrum of $\mls O_X[E]$ over $X$.   Let $\Der_{\mls S}(X, E)$ denote the set of $\mls S$-morphisms $\delta \colon X[E]\rightarrow X$ such that the composition of $\delta $ with the closed immersion $j\colon X\rightarrow X[E]$ is the identity.  Recall here that an $\mls S$-morphism between two $\mls S$-spaces $a\colon U\rightarrow \mls S$ and $b\colon V\rightarrow \mls S$ is a pair $(h, \sigma )$, where $h\colon U\rightarrow V$ is a morphism of algebraic spaces and $\sigma \colon a\rightarrow b\circ h$ is an isomorphism.  

More generally, let 
 $\mls Der _{\mls S}(X, E)$ denote the sheaf on the small \'etale site of $X$ which to any $U\rightarrow X$ associates the set $\Der_{\mls S}(U, E|_U)$.  
The sheaf $\mls Der _{\mls X}(X, E)$ is naturally  a quasi-coherent sheaf.  This can be seen either by studying additional structure on $X[E]$ or by descent, as follows.  If $S\rightarrow \mls S$ is a smooth covering by a scheme, then by descent theory we have 
$$
\Der_{\mls S}(X, E) = \Eq\left(\Der_S\left(X_S, E|_{X_S}\right)\longrightrightarrows \Der_{S\times _{\mls S}S}\left(X_{S\times _{\mls S}S}, E|_{X_{S\times _{\mls S}S}}\right)\right),
$$
where $X_S$ and $X_{S\times _{\mls S}S}$ denote the base changes of $X$, and $\mls Der_X(X, E)$ is the sheaf obtained by descent from $\mls Der _{S}(X_S, E|_{X_S})$ with its natural descent datum.
 The derivation $d\colon \mls O_X\rightarrow \Omega ^1_{X/\mls S}$ can be upgraded to a derivation $\tilde d\in \Der_{\mls S}(X, \Omega ^1_{X/\mls S})$.  Indeed, if we identify $X[\Omega ^1_{X/\mls S}]$ with $X\times _{\mls S}X$ using the first projection, then the second projection defines $\tilde d\in \Der_{\mls S}(X, \Omega ^1_{X/\mls S})$.  By descent theory and the universal property of the derivation $d$ for algebraic spaces, it follows that $\tilde d$ defines an isomorphism
$$
\mls Der_{\mls S}(X, E)\simeq \mls Hom_{\mls O_X}\left(\Omega ^1_{X/\mls S}, E\right).
$$
\end{rem}

\section{Differential graded algebra structure on $\texorpdfstring{\boldsymbol{\mcH _{X/\mls S, *}}$}{HH\textunderscore \{X/S,*\}}}

\subsection{}\label{P:6.1} The sheaves on $\mls C_{X/\mls S}$ described by the Hochschild homology sheaves in Definition~\ref{D:3.8} can be described directly as follows.

First note that every object $(S, U\rightarrow X_S)$ admits a covering by objects with $S$ affine (simply cover $S$ by affines).  Let $\mls C'_{X/\mls S}\subset \mls C_{X/\mls S}$ be the full subcategory of objects with $S$ affine.  Then the topology on $\mls C_{X/\mls S}$ restricts to a topology on $\mls C'_{X/\mls S}$, and the inclusion defines an equivalence of categories of sheaves
$$
\mls C^{\prime \sim }_{X/\mls S}\simeq \mls C_{X/\mls S}^\sim .
$$ 
On $\mls C^{\prime \sim }$ we have a complex of sheaves $\mc C_{X/\mls S}^*$ given by forming the complex $\mc C_{U/R}^*$ of Section~\ref{P:3.2} for each object $(S = \Sp (R), U\rightarrow X_S)$ of $\mls C'_{X/\mls S}$.  The quasi-coherent sheaves $\mcH _{X/\mls S, n}$, viewed as sheaves on $\mls C_{X/\mls S}'$, are then given by $\mls H^{-n}(\mc C_{X/\mls S}^*)$.

For the rest of this section, we view the sheaves $\mcH _{X/\mls S, n}$ as sheaves on $\mls C^{\prime \sim }_{X/\mls S}$ and explain how $\mcH _{X/\mls S, *} = \oplus _n\mcH _{X/\mls S, n}$ has a structure of a commutative differential graded algebra.

\subsection{}   To this end, first recall some basic results about the usual Hochschild homology of a map of rings $k\rightarrow A$:

  \begin{enumerate}
  \item The graded $A$-module $HH_*(A/k)$ has a natural structure of a graded commutative differential graded algebra. 
    The product structure is defined as in \cite[Corollary~4.2.7]{L}, and the differential is given by the Connes differential $B\colon HH_n(A/k)\rightarrow HH_{n+1}(A/k)$.

\item For a commutative diagram of rings
$$
\xymatrix{
A\ar[r]& A'\\
k\ar[r]\ar[u]& k'\ar[u]\rlap{,}}
$$
the induced map
$$
HH_*(A/k)\lra HH_*(A'/k')
$$
is a morphism of differential graded algebras.
  \end{enumerate}
  
It follows by sheafification that the sheaf of graded algebras $\mcH _{X/\mls S, *}$ on $\mls C'_{X/\mls S}$ has a structure of an object of $\dga_{X/\mls S}$.  We refer to the differential $B\colon \mcH _{X/\mls S, n}\rightarrow \mcH _{X/\mls S, n+1}$ as the \emph{Connes differential}, generalizing the scheme case.   

In particular, by the universal property of $\Omega ^\bullet _{X/\mls S}$, we have a unique morphism 
\begin{equation}\label{E:6.2.1}
\epsilon \colon  \Omega ^\bullet _{X/\mls S}\lra \mcH _{X/\mls S, *}
\end{equation}
in $\dga_{X/\mls S}$ extending the natural isomorphism $\epsilon _0\colon \mls O_{X/\mls S}\rightarrow \mcH _{X/\mls S, 0}$.

\begin{prop}[HKR isomorphism]  The map $\epsilon _1\colon \Omega ^1_{X/\mls S}\rightarrow \mcH _{X/\mls S, 1}$ is an isomorphism.  If $f\colon X\rightarrow \mls S$ is smooth, then the map $\epsilon _n\colon \Omega ^n_{X/\mls S}\rightarrow \mcH _{X/\mls S, n}$ is an isomorphism for all $n$.
\end{prop}

\begin{proof}
It suffices to show that corresponding statements for $\Omega ^\bullet _{U/S}\rightarrow \mcH _{U/S, *}$ for an object $(S, U\rightarrow X_S)$ of $\mls C'_{X/\mls S}$, which follow from \cite[Proposition~1.1.10 and Theorem~3.4.4]{L}.
\end{proof}

\begin{rem} One can also construct the maps $\epsilon _n\colon \Omega ^n_{X/\mls S}\rightarrow \mcH _{X/\mls S, n}$ using descent without involving the site $\mls C'_{X/\mls S}$.  However, we prefer to use the site $\mls C'_{X/\mls S}$ since it allows us to obtain the maps $\epsilon _n$ from the universal property of the de Rham complex and the differential graded algebra structure on $\mcH _{X/\mls S, *}$.
\end{rem}

\begin{rem} The approach of Hablicsek, Herr, and Leonardi \cite{hablicsek2023logarithmic} gives a global HKR isomorphism under certain assumptions. 
Let $k$ be a field, and let $f\colon X\rightarrow \mls S$ be a smooth morphism over  $k$.  Assume either that  $k$ has characteristic $0$ or that the characteristic is greater than the relative dimension of $X$ over~$\mls S$.  Then it follows from Theorem~\ref{T:4.2} and \cite[Theorem 0.7 and Section~1.18]{MR2955193} that the Hochschild complex is formal: $\mcH _{X/\mls S}\simeq \oplus _i\Omega ^i_{X/\mls S}[i]$.  Indeed, with notation as in Section~\ref{P:4.1}, since $X/\mls S$ is smooth, the immersion $X\hookrightarrow I$ is a regular immersion of codimension the relative dimension of $X/\mls S$, and either projection $I\rightarrow X$ defines a retraction, so the argument of \cite[Section~1.18]{MR2955193} can be applied, with the additional observation that the result of \textit{loc.~cit.}~also holds for algebraic spaces and not just schemes with the same argument (in general, $I$ is only an algebraic space).
\end{rem}

\begin{example} Let $(X, M_X)\rightarrow (\Sp (k), M_k)$ be a nodal curve with the standard log structure over the log point; see \cite{FKato}.  Then the corresponding map $X\rightarrow \logs _{(\Sp (k), M_k)}$ has relative dimension $1$, and therefore by the preceding remark $\mcH _{(X, M_X)/(\Sp (k), M_k)}\simeq \mls O_X\oplus \Omega ^1_{(X, M_X)/(\Sp (k), M_k)}[1]$ is formal.  On the other hand the Hochschild homology of the underlying scheme $X$ is calculated using the cohomology of the exterior powers of the cotangent complex, as discussed in \cite{he2025hodgerhamdegenerationnodal}, with the map relating the two induced by the natural map $\mathbf{L}_{X/k}\rightarrow \Omega ^1_{(X, M_X)/(\Sp (k), M_k)}$.  The computation for $X/k$ is summarized in \cite[Theorem 5.5]{he2025hodgerhamdegenerationnodal}.
\end{example}

\section{Cyclic homology}\label{S:section7}

\subsection{Generalities}

\subsection{} For a topos $T$ let $T^{\mathbb{N}}$ denote the topos of projective systems indexed by $\mathbb{N}$ of objects of $T$ (see for example \cite[\href{https://stacks.math.columbia.edu/tag/0940}{Tag 0940}]{stacks-project}).  So an object of $T^{\mathbb{N}}$ is a sequence of sheaves 
$$
A_0\longleftarrow A_1\longleftarrow A_2\longleftarrow \cdots . 
$$
For a ring $R$ in $T$, a complex of $R$-modules in ${T}^{\mathbb{N}}$ is given by a double complex of $R$-modules
$$
\xymatrix{
\vdots & \vdots & \vdots & \\
A_0^{q+1}\ar[u]& A_1^{q+1}\ar[u]\ar[l]& A_2^{q+1}\ar[u]\ar[l]& \cdots \ar[l]\\
A_0^{q}\ar[u]& A_1^q\ar[l]\ar[u]& A_2^q\ar[u]\ar[l]& \cdots \ar[l]\\
\vdots \ar[u]& \vdots \ar[u]& \vdots \rlap{.}\ar[u]& }
$$
In what follows we systematically use the notation of referencing the column direction by superscript and the row direction by subscript.
We denote by $C(T^{\mathbb{N}}, R)$ the category of complexes of $R$-modules in $T^{\mathbb{N}}$, and by $D(T^{\mathbb{N}}, R)$ the associated derived category.  Note that a map $A_\bullet ^\bullet \rightarrow B_\bullet ^\bullet $ in $C(T^{\mathbb{N}}, R)$ induces an isomorphism in $D(T^{\mathbb{N}}, R)$ if and only if for every $n$ the map of complexes $A_n^\bullet \rightarrow B_n^\bullet $ is a quasi-isomorphism.

\subsection{} For an integer $n\geq 0$ there is a morphism of topoi
$$
e_n\colon T^\mathbb{N}\lra T.
$$
The functor $e_{n*}$ sends a system of sheaves $A_\bullet $ to $A_n$, and the left adjoint $e_n^*$ sends an object $B\in T$ to the system which is $B$ in degrees at most $n$ and $0$ above $n$:
$$
e_n^*B\colon  \cdots \overset{\id}\longleftarrow B \overset{\id}\longleftarrow  B\longleftarrow  \cdots \longleftarrow B \longleftarrow 0 \longleftarrow \cdots .
$$
For a ring $R$ in $T$, this induces a morphism of ringed topoi
$$
e_n\colon \left(T^{\mathbb{N}}, R\right)\lra (T, R)
$$
and functors on the derived category
$$
Re_{n*}\colon D\left(T^{\mathbb{N}}, R\right)\lra D(T, R), \quad Le_{n}^*\colon D(T, R)\lra D\left(T^{\mathbb{N}}, R\right).
$$

\subsection{}  By reindexing in the usual manner, one obtains from  any object of $C(T^{\mathbb{N}}, R)$ a double complex of $R$-modules concentrated in the second and third quadrants. Taking the product total complex defines a functor 
$$
\int \colon D^-\left(T^{\mathbb{N}}, R\right)\lra D^-(T, R).
$$

\subsection{}\label{P:6.4} If $f\colon (T, R)\rightarrow (T', R')$ is a morphism of ringed topoi, then there is an induced morphism 
$$
f^{\mathbb{N}}\colon \left(T^{\mathbb{N}}, R\right)\lra \left(T^{\prime\mkern1mu \mathbb{N}}, R'\right).
$$
Furthermore, for every integer $n\geq 0$, the diagram 
$$
\xymatrix{
\left(T^{\mathbb{N}}, R\right)\ar[r]^-{f^{\mathbb{N}}}\ar[d]^-{e_n}& \left(T^{\prime\mkern1mu \mathbb{N}}, R'\right)\ar[d]_-{e_n}\\
(T, R)\ar[r]^-f& (T', R')}
$$
commutes.

\subsection{}  Let $X$ be an algebraic space, and let $\pi \colon X_\bullet \rightarrow X$ be a smooth hypercover.  Let $R$ be a sheaf of rings on the big \'etale site of $X$, and let $R_X$ (resp.\ $R_{X_\bullet }$) be the restriction of $R$ to the small \'etale site of $X$ (resp.\ the induced sheaf of rings on $X_{\bullet, \et}$).  We then get an induced morphism of ringed topoi
$$
\left(X_{\bullet, \et}^{\mathbb{N}}, R_{X_\bullet }\right)\lra \left(X_\et^{\mathbb{N}}, R_X\right).
$$
We denote by
$$
R\pi ^{\mathbb{N}}_*\colon D\left(X_{\bullet, \et}^{\mathbb{N}}, R_{X_\bullet }\right)\lra D\left(X_{\et}^{\mathbb{N}}, R_X\right), \quad L\pi ^{\mathbb{N}*}\colon \left(X_{\et}^{\mathbb{N}}, R_X\right)\lra D\left(X_{\bullet, \et}^{\mathbb{N}}, R_{X_\bullet }\right)
$$
the induced functors between derived categories.

\subsection{}
For an object $K\in C(T^{\mathbb{N}}, R)$ and integers $r\geq 0$ and $n\in \mathbb{Z}$, define $K(r, n)$ to be the complex with
$$
(K(r, n))_p^q:= K_{p-r}^{q+n}, \quad p\geq r,
$$
and $K(r, n)_p^q = 0$ for $p<r$.  The functor $(-)(r, n)$ passes to the derived category to give a functor
$$
(-)(r, n)\colon D\left(T^{\mathbb{N}}, R\right)\lra D\left(T^{\mathbb{N}}, R\right).
$$
Observe that for $K\in D^-(T^{\mathbb{N}}, R)$ we have
$$
\int (K(r, n)) = \left(\int K\right)[r+n].
$$

\subsection{Cyclic homology}

\subsection{}\label{P:7.7} Now consider  the situation of Section~\ref{P:3.1}, and let $S_\bullet \rightarrow \mls S$ be a smooth hypercover as in Section~\ref{P:3.3}.  For each $n$ the complex of sheaves $\mc C^*_{X_n/R_n}$ is induced from a cyclic object $\widetilde {\mc C}_{X_n/R_n}$ in the category of sheaves of $k$-modules on $X_{n, \et}$.  By functoriality this defines a cyclic object $\widetilde {\mc C}_{X_\bullet /R_\bullet }$ in the category of $k$-modules in $X_{\bullet, \et}$.  Let $B_{X_\bullet /R_\bullet }$ be the complex of  $k$-modules in $X_{\bullet, \et}^{\mathbb{N}}$ obtained from $\widetilde {\mc C}_{X_\bullet /R_\bullet }$ by the construction of \cite[Section~2.1.7]{L}.  Up to a shift, the columns of $B_{X_\bullet /R_\bullet }$ are given by the Hochschild complex.

\subsection{} Let $B_{X/\mls S, S_\bullet }\in D(X^{\mathbb{N}}_{\et}, k)$ denote $R\pi _*^{\mathbb{N}}B_{X_\bullet /R_\bullet }$. As in the definition of $\mcH _{X/\mls S}$, this is independent of the choice of the hypercover $S_\bullet \rightarrow \mls S$.  More precisely, if $u\colon S'_\bullet \rightarrow S_\bullet $ is a morphism of hypercovers of $\mls S$, then there is an induced morphism
$$
\epsilon _u\colon B_{X/\mls S, S_\bullet }\lra B_{X/\mls S, S'_\bullet },
$$
which is an isomorphism since this can be verified after applying the functor $e_n$, where it follows from the comparison with the Hochschild complex.  Furthermore, by an argument similar to the one proving Proposition~\ref{P:3.6}, one shows that if $u_1$ and $u_2$ are related by a simplicial homotopy, then $\epsilon _{u_1} = \epsilon _{u_2}$.  We can therefore define 
$$
B_{X/\mls S}:= \colim _{S_\bullet \in HR_{\aff}(\mls S)}B_{X/\mls S, S_\bullet }
$$
as in Section~\ref{P:3.7}.

\begin{defn}\label{D:6.9} For an integer $m$ the \emph{$\supth{m}$ cyclic homology} of $X/\mls S$ is defined to be the $k$-module
$$
HC_m(X/\mls S):= H^{-m}\left(\int R\Gamma ^{\mathbb{N}}(X, B_{X/\mls S})\right),
$$
where we write $\Gamma ^{\mathbb{N}}$ for the pushforward for the morphism of ringed topoi
$$
\left(X_{\et}^{\mathbb{N}}, k\right)\lra \left((\pt)^{\mathbb{N}}, k\right)
$$
given by the global section functor.
\end{defn}

\begin{lem} Assume that $X$ is a quasi-compact and quasi-separated algebraic space over $k$.  Then there exists an integer $d$ such that for $n\geq 0$ we have $H^j(e_{n*}R\Gamma ^{\mathbb{N}}_*(B_{X/\mls S})) = 0$ for $j>-n+d$.  In particular, passing to total complexes we have $HC_m(X/\mls S) = 0$ for $m>d$.  If\, $X$ is a Noetherian scheme, then we can take $d$ to be the dimension of\, $X$.
\end{lem}

\begin{proof}
As noted in Section~\ref{P:6.4}, we have $e_{n*}R\Gamma  ^{\mathbb{N}}_*B_{X/\mls S}\simeq R\Gamma  _*e_{n*}B_{X/\mls S}$. By construction we have
$$
e_{n*}B_{X/\mls S}\simeq \mcH _{X/\mls S}[n],
$$
and therefore it suffices to show that there exists an integer $d$ such that  $H^j(R\Gamma_*\mcH _{X/\mls S}) = 0$ for $j>d$.  This follows from noting that $\mcH _{X/\mls S}$ lies in $D^{\leq 0}_{\qcoh}(X_{\et}, \mls O_X)$ and applying \cite[\href{https://stacks.math.columbia.edu/tag/0729}{Tag 0729}]{stacks-project}.
\end{proof}

\begin{rem}  The definition of cyclic homology for a morphism of log schemes $(X, M_X)\rightarrow (\Sp (k), M_k)$  is obtained from the above by setting
$$
HC_m((X, M_X)/(k, M_k)):= HC _m\left(X/\logs _{(k, M_k)}\right).
$$
\end{rem}

\subsection{Comparisons}

\subsection{}
In the above we have worked with the \'etale topology, whereas in \cite{W, WG} the authors work with the Zariski or Nisnevich topology. The equivalence of the two definitions in the classical case of a scheme over a ring $k$ follows from fpqc descent for cyclic homology due 
 to Bhatt, Morrow, and Scholze \cite[Corollary~3.4]{BMS1} (combined with the cofiber sequence relating cyclic homology to negative cyclic homology and periodic cyclic homology; see \cite[Section~5.1.8]{L}).

\subsection{}\label{P:7.16} The complex $B_{X/\mls S}\in D(X_{\et}^{\mathbb{N}}, k)$ can also be described in terms of the topos $\mls C_{X/\mls S}^{\prime \sim }$.  

Namely, for each $(S = \Sp (R), U\rightarrow X_S)\in \mls C_{X/\mls S}'$, let 
$$
B_{(S, U\rightarrow X_S)}
$$
be the complex of $k$-modules in $U_{\et}^{\mathbb{N}}$ obtained from the cyclic object $\widetilde {\mc C}_{U/R}$ as in Section~\ref{P:7.7}.  The complex $B_{(S, U\rightarrow X_S)}$ is functorial in $(S, U\rightarrow X_S)$, and therefore we get a complex $B_{\mls C'_{X/\mls S}}$ in $(\mls C_{X/\mls S}^{\prime \sim })^{\mathbb{N}}$.

\subsection{}
There is a morphism of topoi
$$
q\colon \mls C_{X/\mls S}^\sim \lra X_\et .
$$
The functor $q^{-1}$ sends a sheaf $F$ on $X_\et $ to the sheaf which to any $(S, U\rightarrow X_S)$ associates the global sections $\Gamma (U, F_S)$, where $F_S$ denotes the pullback of $F$ to $X_{S, \et }$.  The functor $q_*$ sends a sheaf $E$ on $\mls C_{X/\mls S}$ to the sheaf on $X_\et $ which sends $U\rightarrow X$ to $\Gamma (\mls C_{U/\mls S}, E|_{\mls C_{U/\mls S}})$.  By the construction of $B_{X/\mls S}$ in Section~\ref{S:section7}, there is a natural map
\begin{equation}\label{E:9.2}
B_{X/\mls S}\lra Rq^\mathbb{N}_*B_{\mls C'_{X/\mls S}}
\end{equation}
in $D(X_\et ^{\mathbb{N}}, k)$.

\begin{lem}
  The map \eqref{E:9.2} is an isomorphism.
\end{lem}

\begin{proof}
It suffices to show that for any $n\in \mathbb{N}$ the map
$$
e_nB_{X/\mls S}\lra e_nRq^{\mathbb{N}}_*B_{\mls C'_{X/\mls S}}\simeq Rq_*e_nB_{\mls C'_{X/\mls S}}
$$ is an isomorphism.  This reduces the proof to the analogous statement for Hochschild homology, which follows from descent theory and Lemma~\ref{L:5.7.1}.
\end{proof}

\section{The SBI sequence}

We continue with the notation of Section~\ref{P:3.3}.

\subsection{}
By the definition of $B_{X_\bullet /R_\bullet }$, we have $e_{0*}B_{X_\bullet /R_\bullet } = \mc C^*_{X_\bullet /R_\bullet }$, so we get an inclusion
$$
e_0^*\mc C^*_{X_\bullet /R_\bullet }\longhookrightarrow B_{X_\bullet /R_\bullet }.
$$
Applying $R\pi _*^{\mathbb{N}}$, we therefore get a distinguished triangle in $D(X_{\et}^{\mathbb{N}}, k)$
$$
e_0^*\mcH _{X/\mls S}\lra B_{X/\mls S}\lra B_{X/\mls S}(1,1)\lra e_0^*\mcH _{X/\mls S}[1].
$$
Taking cohomology and total complexes then induces a long exact sequence
$$
\cdots\lra HH_m(X/\mls S)\lra HC_m(X/\mls S)\lra HC_{m-2}(X/\mls S)\lra HH_{m-1}(X/\mls S)\lra \cdots .
$$
We call this the \emph{SBI sequence}, generalizing \cite[Theorem~2.2.1]{L}.

\subsection{} 
The complex $\int R\Gamma ^{\mathbb{N}}(X, B_{X/\mls S})$ is naturally filtered by columns with associated graded pieces shifts of the Hochschild complex.  As in \cite[Proof of Theorem~3.4.11]{L}, the spectral sequence of a filtered complex therefore gives a spectral sequence abutting to $HC_*(X/\mls S)$ with $E^1_{pq}$-term equal to $HH_{q-p}(X/\mls S)$ for $q\geq p\geq 0$ and $0$ otherwise.

If $X\rightarrow \mls S$ is smooth and $X = \Sp (A)$ is affine, then we can combine this with the HKR isomorphism to get a spectral sequence abutting to $HC_*(X\mls S)$ with 
$$
E^1_{pq} = \begin{cases} \Omega ^{q-p}_{A/\mls S} & \text{if $q\geq p\geq 0$},\\
0 & \text{otherwise}.\end{cases}
$$
As in the classical case, one verifies that the differentials on the $E^1$-page of this spectral sequence are given by the maps $B$, and therefore we get
$$
E^2_{pq} = \begin{cases} \Omega ^q_{A/\mls S}/d\Omega ^{q-1}_{A/\mls S} & \text{if $p=0$},\\[2pt]
H_{\dR}^{q-p}(X/\mls S)& \text{if $p>0$}.\end{cases}
$$

\subsection{}
When $k$ is a $\mathbb{Q}$-algebra, this spectral sequence degenerates.  This can be seen using the same argument as in \cite[Theorem~3.4.12]{L}.

Define $\mc D_{X_\bullet /R_\bullet }\in C(X_\bullet ^{\mathbb{N}}, k)$ to be the complex given by 
$$
\mc D_{X_\bullet /R_\bullet , p}^q = \Omega ^{q-p}_{X_\bullet /R_\bullet },
$$
horizontal differentials given by differentiation, and vertical differentials the zero maps.  Define the complex $\mc D_{X/\mls S}\in C(X^{\mathbb{N}}, k)$ similarly.  Then we have 
$$
\mc D_{X/\mls S}\simeq R\pi _*^{\mathbb{N}}\mc D_{X_\bullet /R_\bullet },
$$
as  can be verified column by column (that is, after applying $e_{n*}$).  The construction of \cite[Section~2.3.6]{L} defines a morphism 
\begin{equation}\label{E:8.3.1}
B_{X_\bullet /R_\bullet }\lra \mc D_{X_\bullet /R_\bullet },
\end{equation}
which induces an isomorphism in $D(X_{\bullet, \et}^{\mathbb{N}}, k)$ if $k$ is a $\mathbb{Q}$-algebra.  Applying $R\pi _*^{\mathbb{N}}$ we get an isomorphism
$$
B_{X/\mls S}\simeq \mc D_{X/\mls S}.
$$
In the case when $X$ equals $\Sp (A)$ and is smooth over $\mls S$, we therefore get an isomorphism
\begin{equation}\label{E:8.3.1b}
HC_m(X/\mls S)\simeq \Omega ^m_{X/\mls S}/d\Omega ^m_{X/\mls S}\oplus \bigoplus _{i\geq 1}H^{m-2i}_{\dR}(X/\mls S).
\end{equation}

\begin{rem} As we explain in Section~\ref{S:section9}, this decomposition can also be described in a more canonical manner using $\lambda $-operations.
\end{rem}

\section{\texorpdfstring{$\boldsymbol{\lambda}$}{lambda}-operations}\label{S:section9}

Cyclic and Hochschild homology carry $\lambda $-operations as in the classical case.  These can be defined either using descent and hypercovers of $\mls S$ or using the site $\mls C'_{X/\mls S}$.  To avoid various technical verifications pertaining to independence of hypercovers, we opt for the latter approach.  We proceed with the notation of Section~\ref{P:7.16}.\looseness=-1

\subsection{} Replacing the Hochschild complexes by their normalizations as in \cite[Section~2.1.9]{L}, we get a quasi-isomorphic quotient complex
$$
B_{\mls C'_{X/\mls S}}\lra \overline B_{\mls C'_{X/\mls S}}.
$$
By the construction in \cite[Section~4.3]{L2}, we then get $\lambda $-operations $\lambda ^k$ on $\overline B_{\mls C'_{X/\mls S}}$ and $\gamma $-filtration $F_\gamma ^\bullet $ on $\overline B_{\mls C'_{X/\mls S}}$.  Applying $Rq_*^{\mathbb{N}}$, this gives $\overline B_{X/\mls S}$ the structure of an object $(\overline B_{X/\mls S}, F^\bullet _\gamma )$ of the filtered derived category $DF(X_\et ^{\mathbb{N}}, k)$.  Furthermore, we get $\lambda $-operations
$$
\lambda ^k\colon \overline B_{X/\mls S}\lra \overline B_{X/\mls S}
$$
in $D(X_\et ^{\mathbb{N}}, k)$, and similarly Adams operations $\psi ^k$ defined as in \cite[Section~1.8]{L2}.

\subsection{}
Similarly, we have $\lambda $-operations defined on Hochschild homology.  Let $\mc C_{X/\mls S}^*$ be the complex on $\mls C'_{X/\mls S}$ defined in Section~\ref{P:6.1}.  By the construction of \cite[Section~3.4]{L2}, this complex comes equipped with $\lambda $-operations and therefore carries a $\gamma $-filtration $F_\gamma ^*$.  Applying $Rq_*$, this gives $\mcH _{X/\mls S}$ the structure of a filtered complex as well as defines $\lambda $-operations and Adams operations $\psi ^k$ on $\mcH _{X/\mls S, *}$.  The Adams operations are compatible with the maps \eqref{E:6.2.1} in the sense that the  diagram
$$
\vspace{-6pt}
\xymatrix{
\Omega ^n_{X/\mls S}\ar[r]\ar[d]^-{\sgn(\psi ^k)}& \mcH _{X/\mls S, n}\ar[d]^-{\psi ^k}\\
\Omega ^n_{X/\mls S}\ar[r]& \mcH _{X/\mls S, n}}
$$
commutes, where $\sgn(\psi^k)$ denotes the signature of $\psi ^k$.  This follows from \cite[Th\'eor\`eme~3.5(d)]{L2}.

\subsection{} If $k$ is a $\mathbb{Q}$-algebra, then using Eulerian
idempotents as in \cite[Section~4.5]{L2}, the complex $\overline B_{\mls C'_{X/\mls S}}$ decomposes into a direct sum 
$$
\overline B_{\mls C'_{X/\mls S}} = \oplus _{i\geq 1}\overline B_{\mls C'_{X/\mls S}}^{(i)}
$$
with $\psi ^k$ acting by multiplication by $k^{i+1}$ on $\overline B_{\mls C'_{X/\mls S}}^{(i)}$.  The $\gamma $-filtration is given in terms of this direct sum decomposition by
$$
F_\gamma ^i = \oplus _{j<i}\overline B_{\mls C'_{X/\mls S}}^{(j)}.
$$
In particular, we have
$$
\gr_\gamma ^i \overline B_{\mls C'_{X/\mls S}} \simeq \overline B_{\mls C'_{X/\mls S}}^{(i-1)}.
$$  Similarly, the complex $\mc C^*_{X/\mls S}$ decomposes as 
$$
\mc C^*_{X/\mls S} = \oplus _{i\geq 0}\mc C^{(i)*}_{X/\mls S}.
$$
We therefore get decompositions
$$
HH _n(X/\mls S) = \oplus _iHH_n(X/\mls S)^{(i)}, \ \ HC_n(X/\mls S) = \oplus _iHC_n(X/\mls S)^{(i)}
$$
into eigenspaces for the Adams operations.

\subsection{} When $X\rightarrow \mls S$ is smooth, we get by the same reasoning as in the proof of \cite[Th\'eor\`eme~4.6]{L2} that the map  \eqref{E:8.3.1} induces a quasi-isomorphism between $\overline B_{\mls C'_{X/\mls S}}^{(i)}$ and the truncated de Rham complex
$$
\Omega ^i_{X/\mls S}\longleftarrow \Omega ^{i-1}_{X/\mls S}\longleftarrow \cdots \longleftarrow \Omega ^0_{X/\mls S}
$$
viewed as a complex in $(\mls C^{\prime \sim }_{X/\mls S})^{\mathbb{N}}$ by placing the complex in the $\supth{i}$ line and putting zeros elsewhere.  From this it follows that the decomposition \eqref{E:8.3.1b} can be obtained in a more canonical manner from the eigenspace decomposition for the Adams operations.


\providecommand{\bysame}{\leavevmode\hbox to3em{\hrulefill}\thinspace}

\end{document}